\documentclass[12pt]{article}

\usepackage[margin=1.25in]{geometry}

\usepackage[english]{babel}
\usepackage[T1]{fontenc}
\usepackage[]{times}
\usepackage{amsmath,amssymb,amsthm,stmaryrd}
\usepackage{bbm}
\usepackage{natbib}
\usepackage{graphicx}
\usepackage{comment}
\usepackage{longtable,dcolumn,booktabs,mathtools}
\usepackage{url}

\usepackage{caption}
\usepackage{graphicx,epsfig,amsmath,amsfonts,amsthm,bbm,setspace,url,multirow}

\newtheorem{theorem}{Theorem}[section]

\theoremstyle{definition}
\newtheorem{definition}[theorem]{Definition}

\newcommand{\real}{{\ensuremath{\mathbb{R}}}}
\newcommand{\sphere}{{\ensuremath{\mathbb{S}}}}

\newcommand{\zz}{{\ensuremath{\mathbb{Z}}}}

\newcommand{\one}{{\ensuremath{\mathbbm{1}}}}

\newcommand{\cB}{{\ensuremath{\cal B}}}

\newcommand{\cP}{{\ensuremath{\cal P}}}

\newcommand{\myE}{{\ensuremath{\mathbb{E}}}}

\newcommand{\myL}{{\ensuremath{\rm L}}}
\newcommand{\myS}{{\ensuremath{\rm S}}}

\newcommand{\hsp}{\hspace{0.2mm}}

\begin{document}

\begin{center}
{\LARGE \bf

On the Cover-Hart Inequality: \\
What's a Sample of Size One Worth? 

}

\bigskip
{\bf Tilmann Gneiting, Universit\"at Heidelberg}

\bigskip
{\bf \today}
\end{center}

\medskip
\begin{abstract}

Bob predicts a future observation based on a sample of size one.
Alice can draw a sample of any size before issuing her prediction.
How much better can she do than Bob?  Perhaps surprisingly, under a
large class of loss functions, which we refer to as the Cover-Hart
family, the best Alice can do is to halve Bob's risk.  In this sense,
half the information in an infinite sample is contained in a sample of
size one.  The Cover-Hart family is a convex cone that includes metrics
and negative definite functions, subject to slight regularity
conditions.  These results may help explain the small relative
differences in empirical performance measures in applied
classification and forecasting problems, as well as the success of
reasoning and learning by analogy in general, and nearest neighbor
techniques in particular.

\end{abstract}

\smallskip
\section{Introduction} \label{sec:intro}

Alice and Bob compete in a game of prediction.  The task is to predict
a future observation, such as a class label, or a real-valued,
vector-valued or highly structured outcome.  Before issuing a point
forecast, Alice and Bob may sample from the underlying population. Bob
has access to a sample of size one only, whereas Alice can draw a
sample of any desired size.  The predictive performance is evaluated
by means of a loss function, $\myL(y,y') \geq 0$, where $y$ is the
point forecast, and $y'$ is the realizing value of the future
observation, $Y'$.  Intuitively, we expect Alice to do much better
than Bob, as she can gather essentially all information available,
thereby attaining or approximating the Bayes risk, namely
\[
\alpha \equiv {\textstyle \inf_y} \, \myE_P \hsp \myL(y,Y'),
\]
where $Y'$ has distribution $P$.  However, Cover and Hart (1967) and
Cover (1968) showed that under misclassification loss and squared
error, Bob's risk, $\beta$, is at most twice Alice's risk, $\alpha$,
that is,
\begin{equation}  \label{eq:CH1}  
\alpha \leq \beta \equiv \myE_P \hsp \myL(Y,Y') \leq 2 \hsp \alpha,
\end{equation} 
where $Y$ and $Y'$ are independent with distribution $P$.

In an elegant and thought-provoking discussion, Cover (1977) noted
that the inequality continues to hold if the loss function is a
metric.  In this paper, we seek a unifying treatment of these
remarkable facts.  Section \ref{sec:point} identifies large classes of
loss functions that satisfy the Cover-Hart inequality (\ref{eq:CH1}),
including both metrics and negative definite functions.  Section
\ref{sec:probabilistic} considers probabilistic predictions, where the
forecasts take the form of predictive probability distributions, and
the predictive performance is evaluated by means of a proper scoring
rule, $\myS(Q,y')$, where $Q$ is the predictive probability
distribution and $y'$ is the realizing observation (Gneiting and
Raftery 2007).  Under the class of kernel scores, which includes the
Brier score and the continuous ranked probability score, an analogue
of the Cover-Hart inequality applies, in that
\begin{equation}  \label{eq:CHP1}  
\alpha \equiv \myE_P \hsp \myS(P,Y') 
\leq \beta \equiv \myE_P \hsp \myS(\delta_Y,Y') \leq 2 \hsp \alpha,
\end{equation} 
where, again, $Y$ and $Y'$ are independent with distribution $P$, and
$\delta_Y$ is the point or Dirac random probability measure in $Y$.
The paper closes with a discussion in Section \ref{sec:discussion},
where we relate to the empirical success of reasoning and learning by
analogy in general, and of nearest neighbor techniques in particular.

\section{Point predictions based on a single observation}  \label{sec:point}

We now discuss the generality of the Cover-Hart inequality.  Toward
this end, we let $\cP$ be the family of the Radon probability measures
on a Hausdorff space $(\Omega, \cB)$, where $\cB$ is the
Borel-$\sigma$-algebra.  We say that a function $\myL: \Omega \times
\Omega \to \real$ is measurable if it is measurable with respect to
either argument when the other argument is fixed.

\begin{definition}  \label{def:CH}
The {\em Cover-Hart class}\/ consists of the measurable functions
$\myL: \Omega \times \Omega \to [0,\infty)$ which are such that
  $\myL(y,y) = 0$ for all $y \in \Omega$, and
\begin{equation}  \label{eq:CH2}   
\alpha \equiv {\textstyle \inf_y} \, \myE_P \hsp \myL(y,Y')  
\leq \beta \equiv \myE_P \hsp \myL(Y,Y') \leq 2 \hsp \alpha
\end{equation} 
for all probability measures $P \in \cP$, where $Y$ and $Y'$ are
independent with distribution $P$.
\end{definition} 

Under a loss function in the Cover-Hart class, half the information in
an infinite sample is contained in a sample of size one, in the sense
that predicting a future observation from a single past observation
incurs at most twice the Bayes risk. 

\begin{theorem}  \label{th:convex.cone}
The Cover-Hart class is a convex cone. 
\end{theorem}

\begin{proof} 
Suppose that $\myL_1$ and $\myL_2$ belong to the Cover-Hart class and
let $c_1, c_2 \geq 0$.  Then the convex combination $\myL = c_1 \myL_1
+ c_2 \myL_2$ is measurable, $\myL(y,y) = 0$ for all $y \in \Omega$,
and
\begin{eqnarray*} 
\myE_P \hsp \myL(Y,Y') & = & c_1 \hsp \myE_P \hsp \myL_1(Y,Y') + c_2 \hsp \myE_P \hsp \myL_2(Y,Y') \\
& \leq & \textstyle 2 \hsp c_1 \inf_{y_1} \myE_P \hsp \myL_1(y_1,Y') 
                  + 2 \hsp c_2 \inf_{y_2} \myE_P \hsp \myL_2(y_2,Y') \\
& \leq & \textstyle 2 \inf_y \left( c_1 \hsp \myE_P \hsp \myL_1(y,Y') 
                                  + c_2 \hsp \myE_P \hsp \myL_2(y,Y') \right) \\
& \leq & \textstyle 2 \inf_y \myE_P \hsp \myL(y,Y')
\end{eqnarray*}
for every $P \in \cP$, whence $\myL$ belongs to the Cover-Hart class.
\end{proof} 

The following result is based on a slight extension of an argument of Cover
(1977), who implicitly assumed the existence of a Bayes rule.

\begin{theorem}  \label{th:metric}
Any measurable metric belongs to the Cover-Hart class.
\end{theorem}

\begin{proof} 
If $\myL$ is a measurable metric, then $\myL$ is nonnegative with
$\myL(y,y) = 0$ and $\myL(y,y'') \leq \myL(y,y') + \myL(y',y'')$ for
all $y, y', y'' \in \Omega$.  Given any $P \in \cP$ there exists a
sequence $(y_n)$ in $\Omega$ such that
\[
\alpha 
= {\textstyle \inf_y} \, \myE_P \hsp \myL(y,Y')  
= {\textstyle \lim_{n \to \infty}} \, \myE_P \hsp \myL(y_n,Y').   
\] 
Thus, 
\[ 
\alpha \leq \beta = \myE_P \hsp \myL(Y,Y') 
\leq \myE_P ( \hsp \myL(Y,y_n) + \myL(y_n,Y') ) 
= 2 \, \myE_P \hsp \myL(y_n,Y') 
\] 
for all integers $n = 1, 2, \ldots$ \ The Cover-Hart inequality
(\ref{eq:CH2}) emerges in the limit as $n \to \infty$, as desired.
\end{proof} 

A function $\myL: \Omega \times \Omega \to [0,\infty)$ is said to be a
  negative definite kernel if it is symmetric in its arguments, with
  $\myL(y,y) = 0$ for all $y \in \Omega$, and
\[
\sum_{i=1}^n \sum_{j=1}^n a_i a_j \, \myL(y_i,y_j) \leq 0
\]
for all finite systems of points $y_1, \ldots, y_n \in \Omega$ and
coefficients $a_1, \ldots, a_n \in \real$ such that $a_1 + \cdots +
a_n = 0$.  Negative definite kernels play major roles in harmonic
analysis (Berg, Christensen and Ressel 1984) and in the theory of
stochastic processes, where they arise as the structure functions or
variograms of random functions with stationary increments (Gneiting,
Sasv\'ari and Schlather 2002).  A wealth of examples of such functions
can be found in the monograph by Berg, Christensen and Ressel (1984)
and the references therein.

\begin{theorem}  \label{th:CND}
Any continuous negative definite kernel belongs to the Cover-Hart class.
\end{theorem}

\begin{proof} 
If $\hsp \myE_P \hsp \myL(y,Y') = \infty$ for all $y$, then clearly the
Cover-Hart inequality (\ref{eq:CH2}) holds.  Thus, we may assume that
$\alpha = \inf_y \myE_P \hsp \myL(y,Y')$ is finite.  By Theorem 2.1 in
Berg, Christensen and Ressel (1984, p.~235),
\[ 
\myE_P \hsp \myL(Y,Y') + \myE_Q \hsp \myL(Z,Z') \leq 2 \, \myE_{Q,P} \hsp \myL(Z,Y'), 
\]
where $P$ and $Q$ are Radon measures, $Y$ and $Y'$ have distribution
$P$, $Z$ and $Z'$ have distribution $Q$, and $Y, Y', Z, Z'$ are
independent.  When $Q$ is the point measure in $y \in \Omega$, the
above inequality implies that $\beta = \myE_P \hsp \myL(Y,Y') \leq 2
\, \myE_P \hsp \myL(y,Y')$ for all $y$, whence the Cover-Hart
inequality is satisfied.
\end{proof} 

\begin{table}[t] 
\centering 
\caption{Examples of negative definite kernels. Here, $\zz$ denotes
  the integers, $\real$ the real numbers, and $\sphere^{d-1}$ the unit
  sphere in the Euclidean space $\real^d$, where $d \geq 2$.  The
  symbol $\one(\cdot)$ stands for an indicator function, $\| \cdot
  \|_p$ for the $\ell_p$-norm or quasi-norm in $\real^d$, and
  $\mbox{gcd}$ for the geodetic or great circle distance on
  $\sphere^{d-1}$.
\label{tab:L}}

\footnotesize
\begin{tabular}{lll}
\toprule
Space & Kernel & Parameters \rule{0mm}{4mm} \\
\midrule
$\zz$          & $\myL(y,y')= \one(y \not= y')$ & \rule{0mm}{4mm} \\
$\real$        & $\myL(y,y') = |y - y'|^q$ & $q \in (0,2]$ \rule{0mm}{4mm} \\
$\real^2$      & $\myL(y,y') = \|y - y'\|_p^q$ & $p \in (2,\infty], \: q \in (0,1]$ \rule{0mm}{4mm} \\
$\real^d$      & $\myL(y,y') = \|y - y'\|_p^q$ \rule{2mm}{0mm} & $p \in (0,2], \: q \in (0,p]$ \rule{0mm}{4mm} \\
$\sphere^{d-1}$ & $\myL(y,y') = \mbox{gcd}(y,y')$ & \rule{0mm}{4mm} \\
\bottomrule
\end{tabular}
\end{table} 

We now discuss a few special cases, which are summarized in Table
\ref{tab:L}.  If $\Omega$ is a discrete space, the misclassification
loss, $\myL(y,y') = \one(y \not= y')$ is a continuous negative
definite kernel.  Thus, Theorem \ref{th:CND} applies and reduces to a
classical result.  When $\Omega$ is finite, the upper bound in the
inequality can be strengthened (Cover and Hart 1967).

If $\Omega$ is the real line, $\real$, the squared error loss
function, $\myL(y,y') = (y-y')^2$ is a continuous negative definite
kernel.  For a far-reaching generalization, let $\| \cdot \|_p$ denote the
standard $\ell_p$-norm or quasi-norm in the Euclidean space $\real^d$.
Schoenberg's theorem (Schoenberg 1938; Berg, Christensen and Ressel
1984, p.~74) and a strand of literature culminating in the work of
Koldobski\v{\i} (1992) and Zastavnyi (1993) demonstrate that the
kernel
\[
\myL(y,y') = \| y - y' \|_p^q
\]
is negative definite under the conditions stated in Table \ref{tab:L},
but not otherwise.  Theorem \ref{th:CND} applies and the respective
loss function is a member of the Cover-Hart class.  To give an
explicit example, if $m = 1$ and the probability measure $P$ is
Gaussian, then $\alpha = 2^{\hsp q/2} \hsp \beta$.

Negative definite kernels can readily be constructed from positive
definite functions (Schoenberg 1938; Gneiting, Sasv\'ari and Schlather
2002).  In this light, graph kernels (Borgwardt et al.~2005;
Vishwanathan et al.~2010) and related types of positive definite
functions on discrete structured spaces, as reviewed by Hofmann,
Sch\"olkopf and Smola (2008), yield Cover-Hart loss functions that are
relevant to the prediction of highly structured objects, such as
strings, trees, graphs and patterns.

\section{Probabilistic predictions based on a single observation}  \label{sec:probabilistic}

Thus far, we have studied single-valued point forecasts.  In this
section, we turn to probabilistic predictions, where the forecasts
take the form of predictive probability distributions over future
quantities and events (Gneiting 2008).  Technically, we retain the
above setting and let $\cP$ denote the class of the Radon probability
measures on a Hausdorff space $(\Omega, \cB)$.  Predictive performance
is evaluated by means of a score,
\[
\myS(Q,y'), 
\]
that quantifies the loss when the probabilistic forecast is the Radon
probability measure $Q \in \cP$, and the realizing observation is $y'
\in \Omega$.  

A scoring rule thus is a function $\myS : \cP \times \Omega \to
\real$.  It is called proper if
\[ 
\myE_P \hsp \myS(P,Y') \leq \myE_P \hsp \myS(Q,Y') 
\] 
for all probability measures $P, Q \in \cP$, where $Y'$ has
distribution $P$ and the expectations are assumed to exist.  In other
words, proper scoring rules encourage careful and honest probabilistic
predictions and prevent hedging.

By Theorem 4 of Gneiting and Raftery (2007), proper scoring rules can
be constructed from negative definite kernels, as follows.\footnote{As
  pointed out to the author by Jochen Fiedler, Theorem 4 of Gneiting
  and Raftery (2007) ought to be formulated relative to the class of
  the Radon probability measures on $\Omega$, as opposed to the larger
  class of the Borel probability measures.  While we are unaware of a
  counterexample for Borel measures, the result of Berg, Christensen
  and Ressel (1984) used in the proof of the theorem applies to Radon
  measures only.}  Let $\myL$ be a nonnegative, continuous negative
definite kernel.  Then the scoring rule
\[
\myS(P,y') = \myE_P \hsp \myL(Y,y') - \frac{1}{2} \, \myE_P \hsp \myL(Y,Y')
\]
is proper relative to the class of the Radon probability measures $P$
for which the expectation $\myE_P \hsp \myL(Y,Y')$ is finite, where
$Y$ and $Y'$ are independent with distribution $P$.  Scoring rules of
this form are referred to as kernel scores, and several of the most
popular and most frequently used examples belong to this class,
including both the Brier score and the continuous ranked probability
score.
 
Under a kernel score, a straightforward calculation leads to a natural
analogue of the Cover-Hart inequality that applies to probabilistic
predictions.  Specifically, if we define $\alpha \equiv {\textstyle
\inf_Q} \, \myE_P \hsp \myS(Q,Y')$ and $\myS$ is a kernel score, a
straightforward calculation shows that
\begin{equation}  \label{eq:CHP2}  
\alpha = \myE_P \hsp \myS(P,Y') 
\leq \beta \equiv \myE_P \hsp \myS(\delta_Y,Y') = 2 \hsp \alpha,
\end{equation} 
where $Y$ and $Y'$ are independent with distribution $P$, and
$\delta_Y$ is the point or Dirac random probability measure in $Y$.
Again, half the information in an infinite sample is contained in a
sample of size one, in that probabilistically predicting a future
observation from a single past observation incurs at most twice the
Bayes risk.

\section{Discussion} \label{sec:discussion}

Despite being well known in pattern analysis and information theory
(see, for example, Devroye, Gy\"{o}rfi and Lugosi 1996), the ground
breaking work of Cover and Hart (1967) and Cover (1968) has hardly
received any attention in the statistical literature.

In this paper, we have demonstrated that the Cover-Hart inequality
(\ref{eq:CH1}) applies whenever the loss function is a measurable
metric, or a continuous negative definite kernel.  Many but not all
metrics are negative definite (Meckes 2011), and so the two families
may have distinct members.  An interesting open question is whether or
not the Cover-Hart class equals the convex cone that is generated by
these two families.  In particular, I do not know whether or not the
Cover-Hart class contains any asymmetric loss functions.

Typically, predictions are conditional on an information set, leading
to natural ramifications of single nearest neighbor methods, such as
nonparametric regression (Stone 1977) and kernel estimators of
conditional predictive distributions (Hyndman, Bashtannyk and Grunwald
1996; Hall, Wolff and Yao 1999).  While we have suppressed the
dependence on the information set in our work, the Cover-Hart
inequality remains valid in this setting, by conditioning on and
integrating over the information set.

In this light, if empirically observed mean score differentials exceed
100\%, this may suggest that forecasters have distinct information
sets.  A simulation example is reported on in Tables 4 and 6 of
Gneiting (2011), where the differences in the predictive scores
between Mr.~Bayes and his competitors, whose predictions are based on
thoroughly distinct information sets, are striking.

From an applied perspective, the Cover-Hart inequality (\ref{eq:CH1})
for point forecasts, and its analogue (\ref{eq:CHP1}) for
probabilistic forecasts, allow for interesting interpretations.  Given
that under many of the most prevalent loss functions used in practice,
Alice, despite having an infinite sample at her disposal, can at most
halve Bob's risk, who has access to a sample of size one only, it is
not surprising that empirically observed differentials in the
predictive performance of competing forecasters tend to be small.  For
example, this was observed in the Netflix contest, where predictive
performance was measured in terms of the (root mean) squared error
(Feuerverger, He and Khatri 2012).  Taking a much broader perspective,
the Cover-Hart inequality may contribute to our understanding of the
empirical success not only of nearest neighbor techniques and their
ramifications, but reasoning and learning by analogy in general
(Gentner and Holyoak 1997).

\section*{Acknowledgements}

The author is grateful to the Alfried Krupp von und zu Behlen
Foundation for support, and thanks Jochen Fiedler, Marc Genton,
Christoph Schn\"{o}rr and Jon Wellner for discussions and references.

\section*{References}

\newenvironment{reflist}{\begin{list}{}{\itemsep 0mm \parsep 1mm
\listparindent -7mm \leftmargin 7mm} \item \ }{\end{list}}

\vspace{-7mm}
\begin{reflist}

Berg, C., Christensen, J.~P.~R., and Ressel, P.~(1984). {\em Harmonic
Analysis on Semigroups}.  New York: Springer.

Borgwardt, K.~M., Ong, C.~S., Sch\"onauer, S., Vishwanathan, S.~V.~N.,
Smola, A.~J.~and Kriegel, H.-P.~(2005).  Protein function prediction
via graph kernels.  {\em Bioinformatics} 21, {i}47--{i}56.

Cover, T.~M.~(1968).  Estimation by the nearest neighbor rule.  {\em
  IEEE Transactions on Information Theory}, 14, 50--55.

Cover, T.~M.~(1977). Comment. {\em Annals of Statistics}, 5, 627--628. 

Cover, T.~M.~and Hart, P.~E.~(1967).  Nearest neighbor pattern
classification.  {\em IEEE Transactions on Information Theory}, 13,
21--27.

Devroye, L., Gy\"{o}rfi, L., and Lugosi, G.~(1996).  A Probabilistic
Theory of Pattern Recognition.  Springer, New York. 

Feuerverger, A., He, Y.~and Khatri, S.~(2012).  Statistical
significance of the Netflix challenge.  {\em Statistical Science}, in
press.

Gentner, D.~and Holyoak, K.~J.~(1997).  Reasoning and learning by
analogy: Introduction.  {\em American Psychologist}, 52, 32--34.

Gneiting, T.~(2008).  Editorial: Probabilistic forecasting.  {\em
  Journal of the Royal Statistical Society Series A: Statistics in
  Society}, 171, 319--321.

Gneiting, T.~(2011).  Making and evaluating point forecasts.  {\em
  Journal of the American Statistical Association}, 106, 746--762.

Gneiting, T.~and Raftery, A.~E.~(2007).  Strictly proper scoring
rules, prediction, and estimation.  {\em Journal of the American
  Statistical Association}, 102, 359--378.

Gneiting, T., Sasv\'ari, Z.~and Schlather, M.~(2001).  Analogies and
correspondences between variograms and covariance functions.  {\em
  Advances in Applied Probability}, 33, 617--630.

Hall, P., Wolff, R.~C.~L.~and Yao, Q,~(1999).  Methods for estimating
a conditional distribution function.  {\em Journal of the American
  Statistical Association}, 94, 154--163.

Hofmann, T., Sch\"olkopf, B.~and Smola, A.~(2008).  Kernel methods in
machine learning.  {\em Annals of Statistics}, 36, 1171--1220.

Hyndman, R.~J., Bashtannyk, D.~M.~and Grunwald, G.~K.~(1996).
Estimating and visualizing conditional densities.  {\em Journal of
  Computational and Graphical Statistics}, 5, 315--336.

Koldobski\v{\i}, A.~L.~(1992).  Schoenberg's problem on positive
definite functions.  {\em St.~Petersburg Mathematical Journal}, 3,
563--570.

Meckes, M.~W.~(2011).  Positive definite metric spaces.  Preprint, 
\url{arxiv:1012.5863v3.pdf}. 

Schoenberg, I.~J.~(1938).  Metric spaces and positive definite
functions.  {\em Transactions of the American Mathematical Society},
44, 522--536.

Stone, C.~J.~(1977).  Consistent nonparametric regression (with
discussion).  {\em Annals of Statistics}, 5, 595--645.

Vishwanathan, S.~V.~N., Schraudolph, N., Kondor, R., and Borgwardt,
K.~M.~(2010).  Graph kernels.  {\em Journal of Machine Learning
  Research} 11, 1201--1242.

Zastavnyi, V.~P.~(1993).  Positive definite functions depending on the
norm.  {\em Russian Journal of Mathematical Physics}, 1, 511--522.
\end{reflist}

\end{document}